\documentclass[a4paper,12pt]{article}
\usepackage{amsfonts}
\usepackage{pifont}
\usepackage{caption}
\usepackage{graphicx, subfig}
\usepackage{bm}
\usepackage{latexsym,amsmath,amssymb,cite,amsthm}
\usepackage{color,eucal,enumerate,mathrsfs}
\usepackage[normalem]{ulem}
\usepackage{amsmath}
\usepackage[pagewise]{lineno}
\usepackage[pagebackref=true, colorlinks, linkcolor=blue,  anchorcolor=blue,citecolor=blue]{hyperref}

\numberwithin{equation}{section}
\theoremstyle{plain}
\newtheorem{theorem}{Theorem}[section]
\newtheorem{corollary}[theorem]{Corollary}
\newtheorem{lemma}[theorem]{Lemma}
\newtheorem{example}[theorem]{Example}
\newtheorem{proposition}[theorem]{Proposition}
\theoremstyle{definition}
\newtheorem{definition}[theorem]{Definition}
\theoremstyle{remark}

\newtheorem{remark}[theorem]{Remark}

\newcommand\bdf{\begin{definition}}
\newcommand\bpr{\begin{proposition}}
\newcommand\brk{\begin{remark}}
\newcommand\blm{\begin{lemma}}
\newcommand\bexe{\begin{exercise}}
\newcommand\bexa{\begin{example}}
\newcommand\beqn{\begin{eqnarray*}}
\newcommand\edf{\end{definition}}
\newcommand\epr{\end{proposition}}
\newcommand\erk{\end{remark}}
\newcommand\elm{\end{lemma}}
\newcommand\eexe{\end{exercise}}
\newcommand\eexa{\end{example}}
\newcommand\eeqn{\end{eqnarray*}}

\newcommand{\Opt}{\rm OptGeo}
\newcommand{\geo}{\rm Geo}
\newcommand{\pr}{\mathcal{P}}

\newcommand{\lmt}[2]{\mathop{\lim}_{{#1} \rightarrow {#2}} }

\newcommand{\lip}[1]{{\mathrm{lip}}({#1})}

\newcommand{\lmti}[2]{\mathop{\underline{\lim}}_{{#1} \rightarrow {#2}} }

\newcommand{\Ric}{{\mathrm{Ricci}}}

\newcommand{\mm}{\mathfrak m}
\newcommand{\ms}{(X,\dist,\mm)}
\newcommand{\CD}{{\rm CD}}

\newcommand{\cdkn}{\mathrm {CD}(K, N)}

\newcommand{\rcdkn}{\mathrm {RCD}(K, N)}

\newcommand{\ent}[1]{\mathrm {Ent}_{#1}}
\newcommand{\De}{\mathbf{D}}
\newcommand{\K}{\mathrm{K}}

\newcommand{\N}{\mathbb{N}}

\newcommand{\R}{\mathbb{R}}
\newcommand{\Rem}{\mathbf{R}}

\newcommand{\supp}{\mathop{\rm supp}\nolimits}   %%\newcommand{\span}{\mathop{\rm span}\nolimits}   %
\newcommand{\Lip}{\mathop{\rm Lip}\nolimits}

\newcommand{\dist}{\mathsf{d}}
\renewcommand{\d}{{\mathrm d}}

\newcommand{\D}{{\mathrm D}}
\newcommand{\restr}[1]{\lower3pt\hbox{$|_{#1}$}}

\newcommand{\nchi}{{\raise.3ex\hbox{$\chi$}}}
\newcommand{\weakto}{\rightharpoonup}

\title{{\bf ABP estimate on metric measure spaces via optimal transport}}

\begin{document}
\author{Bang-Xian Han\thanks{School of Mathematics, Shandong University.  Email: hanbx@sdu.edu.cn. }}

\date{\today} 
\maketitle

%%%%%%%%%%%%%%
% 摘要还需要更加详细
%%%%%%%%%%%%%%
\begin{abstract}
By using  optimal transport theory,  we establish a sharp Alexandroff--Bakelman--Pucci (ABP) type estimate  on  metric measure spaces with  synthetic Riemannian Ricci curvature lower bounds, and prove some geometric  and functional inequalities including a functional  ABP estimate.    Our result not only extends the border of  ABP estimate,  but also provides an  effective substitution of Jacobi fields computation  in the non-smooth framework, which   has potential applications to many problems in non-smooth geometric analysis. 
\end{abstract}

\tableofcontents

\section{Introduction}
During the sixties, Alexandroff, Bakelman, and Pucci introduced a method, which we call  ABP method today, to prove  ABP estimate.  It plays a key role in the proof of the Krylov\textendash Safonov Harnack inequality and the regularity theory for fully non-linear elliptic equations,  see \cite{Cabrelecture, CabreCaffarelli} for an introduction to this important theory and  a beautiful proof for the  isoperimetric inequality for smooth domains in  the Euclidean space. 

Formally speaking,  the goal of the Alexandroff\textendash Bakelman\textendash Pucci (ABP) estimate  is to \emph{estimate the size of the contact sets}, which  naturally has a non-linear version. 
The study of the ABP estimate on Riemannian manifolds was initiated by Cabr\'e  in \cite{Cabre-ABP} where  he considered the square of distance functions (or concave paraboloid) instead of affine functions as the touching functions, and obtained the Harnack inequalities for non-divergent elliptic equations on Riemannian manifolds with non-negative sectional curvature. Based on Cabr\'e's idea and a work of Savin \cite{SavinCPDE},  Wang and Zhang \cite{WZ-ABP}  introduced a notion of contact set instead of the convex envelope, and established an explicit ABP type estimate on Riemannian manifolds with Ricci curvature bounded from below.  More recently,  Xia and Zhang \cite{XZ-ABP}  established an anisotropic version of the ABP estimate, and proved several geometric inequalities using this estimate. 

In the recent achievement of Gigli \cite{Gigli2023On} and Mondino--Semola\cite{mondino2022lipschitz} on the regularity theory for harmonic maps from $\rcdkn$ to ${\rm CAT}(0)$ spaces,  a non-sharp version of this type of estimate plays important role. Given also its importance in both elliptic equations and geometry,  we therefore believe that a sharp version of ABP estimate on more general metric measure spaces has its own interest.

In the present work, we continue the study of ABP type estimate on non-smooth metric measure spaces. In particular, we shall establish a sharp version of the ABP estimate on metric measure spaces with Ricci curvature bounded from below.   

\medskip

 As pointed by Mondino and Semola in \cite[\S 4]{mondino2022lipschitz}, there are
 \begin{quote}
\emph{`a couple of deep difficulties to repeat  strategy of Wang\textendash Zhang \cite{WZ-ABP} in the non-smooth setting'}. 
 \end{quote}
Precisely,  the difficulties are:
\begin{itemize}
\item  In the non-smooth setting, Jacobi fields computations are not available and typically one works with Wasserstein geodesics  in order to take advantage of optimal transport tools.
\item  An initial value problem (an ODE which play a key role in the argument)  has no clear counterpart in the non-smooth setting.
\end{itemize}

To overcome these difficulties, we will make full use of the powerful  optimal transport techniques developed in the last decade.
Optimal transport, or called optimal mass transportation,  aims to evaluate the difference between  two probability measures.
 For  $p\geq 1$, $\mathcal P_p(X)$ denotes the set of probability measures on a metric space $(X, \dist)$ with finite $p$-moment, i.e. $\mu\in \mathcal P_p(X)$ if $\mu(X)=1$ and $\int \dist^p(x, x_0)\,\d \mu(x)<\infty$ for some 
 (and thus every) $x_0\in X$. The $L^p$-transport distance, or called $p$-Wasserstein distance $W_p$,  is defined by
 \[
 W_p^p(\mu, \nu):=\inf_\Pi \int \dist^p(x, y)\,\d \Pi(x, y)
 \]
 where the infimum is taken among all transport plans $\Pi$ with marginals $\mu, \nu\in \mathcal P_p(X)$.
 
Take $p=2$ for example, it is known that $W_2$ can be computed by  duality 
\[
 \frac 12  W_2^2(\mu, \nu)=\mathop{\sup}_{(\varphi, \phi)} \, \left\{\int \varphi(x)\,\d \mu(x)+\int \phi(y)\,\d \nu(y)\right\}
\]
where the supremum is taken over all pairs of integrable functions $(\varphi, \phi)$ satisfying $$\varphi(x)+\phi(y)\leq \frac {\dist^2(x, y)}2~~~~~\forall x, y\in X.$$
Equivalently,  we can consider all pairs of functions $(\varphi, \phi)$ with $\varphi \in \Lip(\supp \mu, \dist)$ and
\[
\phi(y)=\varphi^c(y):=\mathop{\inf}_{x\in \supp \mu} \left(\frac{\dist^2(x, y)}2-\varphi(x)\right)~~~~~\forall y\in \supp \nu.
\]
Note that $\phi$ defined as above is locally Lipschitz on $\supp \nu$. By optimal transport theory,
there is a locally Lipschitz function $\varphi$, called Kantorovich potential, such that 
\[
\frac 12 W^2_2(\mu, \nu)=\int \varphi(x)\,\d \mu(x)+\int \varphi^c(y)\,\d \nu(y).
\]

Consider the following problem, which can be seen as an {\bf inverse problem} of the optimal transport problem:
\begin{quote}
\emph{Given a function $\varphi$, can we find a  pair of  probability measures,  with maximal supports,  so that  $\varphi$ is a Kantorovich potential associated with the corresponding optimal transport problem?} 
\end{quote}
We will see that this inverse problem, together with the curvature-dimension condition,  is the  essence of the ABP estimate. In particular, we will see that \emph{different optimal transport problems correspond to different contact sets}.

\medskip

For  $L^2$-optimal transport problem,  following Cabr\'e   \cite{Cabre-ABP} and Wang--Zhang \cite{WZ-ABP},  we  need to consider the following $2$-contact set.

\bdf[Contact set $\Rem_2$]\label{def:transportset2}
Let  $\Omega$ be a bounded open subset of $X$ and $u$ be a  continuous function on $X$.  For a given $t>0$ and a compact set $\De\subset X$, we define the $2$-contact set $ \Rem_2(\De, \Omega,  u, t)$ associated to  $u$ of opening $t$ with vertex set $\De$ by
 \begin{equation*}
 \Rem_2(\De, \Omega, u, t):=\left\{ x\in \overline \Omega: \exists y\in \De~\text {s.t.} ~\inf_{\overline \Omega}\Big(u+\frac { \dist^2_y}{2t}\Big)=u(x)+\frac { \dist^2(x, y)}{2t}\right\}
 \end{equation*}
 where  $\dist_y(\cdot):=\dist(\cdot, y)$ the distance function to a point $y\in X$.
\edf

\medskip

In $L^1$-optimal transport, the mass will be transported along the trajectories of the gradient of a Kantorovich potential $\phi$, each pair of  points $x, y$ on the same trajectory satisfies $\dist(x, y)=|\phi(x)-\phi(y)|$. 
So we  need to  fix the distance between $x, y$ and consider   the $1$-contact set $\Rem_1$ in a different way.

\bdf[Contact sets $\Rem^*_1$ and $\Rem_1$]\label{def:transportset1}
Let $\Omega$ be a bounded open subset of $X$. For a given continuous function $u$,  a compact set $\De\subset X$ and $t\geq 0$,  the contact set $ \Rem_1(\De, \Omega,  u, t)$ associated with  $u$ of opening $t$,  with vertex set $\De$, is defined by
 \begin{equation*}
 \Rem_1(\De, \Omega,  u, t):=\left\{ x\in \overline \Omega: \exists y\in \De~\text {s.t.}  ~\dist(x, y)=t, \inf_{\overline \Omega}\big(u+\dist_y\big)=u(x)+\dist(x, y)\right\}.
 \end{equation*}
 We also denote 
 \begin{eqnarray*}
  \Rem^*_1(\De, \Omega,  u)=\left\{ x\in \overline \Omega: \exists y\in \De~~\text {s.t.} ~~\inf_{\overline \Omega}\big(u+\dist_y\big)=u(x)+\dist(x, y)\right\}.
 \end{eqnarray*}
\edf

\medskip

\begin{remark}
The bold letters $\Rem$ and $\De$ come from French words \emph{remblais} and \emph{d\'eblais} respectively,   which appear    in the title of the well-known article  \cite[M\'emoire sur la th\'eorie des {\bf D}\'eblais et de {\bf R}emblais]{monge1781} published in  1781.  This article is the starting point of the optimal transport theory,  written by    French mathematician Gaspard Monge (1746-1818).
\end{remark}
\medskip

Using synthetic curvature-dimension theory initiated by Lott--Villani \cite{Lott-Villani09} and Sturm \cite{S-O1, S-O2}, and  non-smooth  calculus tools developed by Gigli \cite{G-O}, in Subsection \ref{l2} we  extend the classical ABP estimate, and Wang--Zhang's estimate \cite{WZ-ABP} to metric measure spaces with Riemannian Ricci curvatures bounded from below. This improves the recent estimates obtained by Gigli \cite[Theorem 5.9 ]{Gigli2023On} and Mondino--Semola \cite[Theorem 4.3]{mondino2022lipschitz}, to a sharp version with explicit constants.

\medskip

\begin{theorem}[ABP estimate]\label{th2}
Let $\ms$ be an $\rcdkn$ metric measure space with $K\in \R$ and $N\in (1, +\infty)$.  Let $\Omega\subset X$ be a bounded open set with $\mm(\partial \Omega)=0$, $\De \subset X$ be a compact set.  Assume there is   a  continuous function  $u\in {\rm D}(\Delta, \Omega)$ with $\Delta u\in L^\infty$ and $t>0$ such that
\[
 \Rem_2(\De, \Omega,  u, t) \subset \Omega
\]
or 
\[
  \Rem_1(\De, \Omega,  u, t) \subset \Omega,
\]
\[
\forall y\in \De,  ~\exists x\in  \Rem_1(\De, \Omega,  u, t), ~\dist(x, y)=t, ~\inf_{\overline \Omega}\big(u+\dist_y\big)=u(x)+\dist(x, y)
\]
Then  for $i=1, 2$, we have
\[
\mm(\De)\leq
\begin{cases}
\displaystyle \mm(\Rem_i)\left(c_{K/N}(\Theta)+\frac{ts_{K/N}(\Theta)}{N\Theta} \|(\Delta u)^+\|_{L^\infty(\overline \Omega)}\right)^N & \textrm{if}\ K<0, \crcr
\displaystyle \mm(\Rem_i)\left(1+\frac tN \|(\Delta u)^+\|_{L^\infty(\overline \Omega)}\right)^N & \textrm{if}\ K =0,  \crcr
\displaystyle   \mm(\Rem_i)\left(c_{K/N}(\Phi)+\frac{ts_{K/N}(\Phi)}{N\Phi}  \|(\Delta u)^+\|_{L^\infty(\overline \Omega)}\right)^N & \textrm{if}\ K>0.
\end{cases}
\]
where $(\Delta u)^+$ denotes the positive part of $\Delta u$,  $\Theta:=\sup_{(x, y)\in \De \times \Omega} {\dist(x, y)}$ and $\Phi:=\inf_{(x, y)\in \De \times \Omega} {\dist(x, y)}$,  $c_{K/N}$ and  $s_{K/N}$ are distortion coefficients.

In particular, if $K=0$, we have
\[
\mm(\De ) \leq  \mm(\Rem_i)\exp( t \|(\Delta u)^+\|_{L^\infty}),~~i=1,2.
\]

\end{theorem}
\bigskip

This theorem  will be proved in   general (possibly non-smooth)
 metric measure spaces $(X,\dist,\mm)$,  satisfying the synthetic condition $\rcdkn$ of Lott--Sturm--Villani \cite{Lott-Villani09,S-O1, S-O2}.  Here  $K\in \R $ denotes Ricci curvature lower bound  and $N\in (1, +\infty)$ denotes dimension upper bound.

\begin{example}[Notable examples of spaces fitting our framework]
The class of ${\rm RCD}(K,N)$ spaces includes the following remarkable subclasses:
\begin{itemize}
\item Measured Gromov--Hausdorff limits of  $N$-dimensional Riemannian manifolds  with  $\mathrm {Ricci} \geq K$,  see \cite{AGS-M}. 
\item $N$-dimensional Alexandrov spaces with curvature bounded from below by $K$, see  \cite{ZhangZhu10, Petrunin11}. 
\end{itemize}
We refer the readers to Villani's Bourbaki seminar \cite{VillaniJapan} and Ambrosio's  ICM-Proceeding \cite{AmbrosioICM} for  more examples and bibliography.
\end{example}

\bigskip
\noindent{\bf Final remarks:}
\begin{itemize}
\item  With the help of  a non-smooth version of Otto's calculus developed by Gigli in \cite{G-O},  we get a non-smooth version of  ABP type estimate without any `Jacobi fields computation'.   This improves an estimate obtained by Gigli   \cite{Gigli2023On} and Mondino--Semola\cite{mondino2022lipschitz}.
\item In Proposition \ref{lemmaEVI}, we prove a functional  version of  ABP estimate, even without the `essentially non-branching' condition,  which seems new even on $\R^n$. 

\item Our results are essentially dimension-dependent, see \cite{Gigli2023On} for a non-sharp, but  dimension-free version.
\end{itemize}
\bigskip
\noindent{\bf Organization of the paper:}
The paper is organized as follows: In Section \ref{pre}, we collect some preliminaries about the theory of metric measure space, optimal transport and curvature-dimension condition. Section \ref{main} is devoted to proving the main theorems and their applications.  

\bigskip

\noindent \textbf{Declaration:}
{The  author declares that there is no conflict of interest and the manuscript is purely theoretical  which has no associated data.}

\medskip

\noindent \textbf{Acknowledgement}:   This  work is supported by  the Young Scientist Programs  of the Ministry of Science \& Technology of China (No. 2021YFA1000900  and 2021YFA1002200),  and NSFC grant (No.12201596).

\section{Preliminaries}\label{pre}
In this paper,  $(X, \dist)$ represents a complete, proper and separable geodesic space endowed with a  positive Radon measure $\mm$ with full support.  The triple   $\ms$ is called a metric measure space.

\subsection{Optimal transport and curvature-dimension condition}\label{SS:CDDef}

\subsubsection*{Metric space and Wasserstein space}
We denote by 
$$
\geo(X) : = \Big\{ \gamma \in C([0,1], X):  \dist(\gamma_{s},\gamma_{t}) = |s-t| \dist(\gamma_{0},\gamma_{1}), \text{ for every } s,t \in [0,1] \Big\}
$$
the space of constant speed geodesics. The metric space $(X,\dist)$ is assumed to be {geodesic}, this means,  for each $x,y \in X$ 
there is $\gamma \in \geo(X)$ so that $\gamma_{0} =x, \gamma_{1} = y$.

\medskip

We denote with  $\mathcal{P}(X)$ the  space of all Borel probability measures over $X$ and with  $\mathcal{P}_{2}(X)$ the space of probability measures with finite second moment.
The $2$-Wasserstein distance  $W_{2}$ is defined as follows:  for $\mu_0,\mu_1 \in \mathcal{P}_{2}(X)$,  set
\begin{equation}\label{eq:W2def}
  W_2^2(\mu_0,\mu_1) := \inf_{ \Pi} \int_{X\times X} \dist^2(x,y) \, \d\Pi(x, y),
\end{equation}
where the infimum is taken over all $\Pi \in \mathcal{P}(X \times X)$ with $\mu_0$ and $\mu_1$ as the first and the second marginal. The space of all measures achieving the minimum in \eqref{eq:W2def} will be denoted by ${\rm Opt}(\mu_0, \mu_1)$ and any $\Pi\in {\rm Opt}(\mu_0, \mu_1)$ will be called \emph{optimal transport plan}.

\medskip

 For any $t\in [0,1]$,  let ${\rm e}_{t}$ denote the evaluation map: 
$$
  {\rm e}_{t} : \geo(X) \to X, \qquad {\rm e}_{t}(\gamma) : = \gamma_{t}.
$$
The space  $(\mathcal{P}_2(X), W_2)$ is geodesic and any geodesic $(\mu_t)_{t \in [0,1]}$ in $(\mathcal{P}_2(X), W_2)$  can be lifted to a measure $\pi \in {\mathcal {P}}(\geo(X))$, 
so that $({\rm e}_t)_\sharp \, \pi = \mu_t$ for all $t \in [0,1]$. 
Given $\mu_{0},\mu_{1} \in \mathcal{P}_{2}(X)$, we denote by 
$\Opt(\mu_{0},\mu_{1})$ the space of all $\pi \in \mathcal{P}(\geo(X))$ for which $({\rm e}_0,{\rm e}_1)_\sharp\, \pi \in {\rm Opt}(\mu_0, \mu_1)$. Such a $\pi$ will be called \emph{dynamical optimal transport plan}. The set  $\Opt(\mu_{0},\mu_{1})$ is non-empty for any $\mu_0,\mu_1\in \mathcal{P}_2(X)$.

\medskip

\subsubsection*{Fundamental theorem of optimal transport}

It is known that $W_2$ can be computed with the following Kantorovich duality formula
\[
 \frac 12  W_2^2(\mu, \nu)=\mathop{\sup}_{(\varphi, \varphi^c)}  \left\{\int \varphi(x)\,\d \mu(x)+\int \varphi^c(y)\,\d \nu(y)\right\}
\]
where the supremum is taken over all pairs of Lipschitz functions $\varphi$ and its $c$-transform  
\[
\varphi^c(y):=\mathop{\inf}_{x\in X} \frac{\dist^2(x, y)}2-\varphi(x)~~~\forall y\in X.
\]

A function $\phi: X \mapsto \R \cup \{-\infty\}$ is called $c$-concave provided it is not identically $-\infty$ and it holds $\phi = \psi^c$ for some $\psi$.
By optimal transport theory,
there is a $c$-concave function $\varphi$, called \emph{Kantorovich potential}, such that 
\[
\frac 12 W^2_2(\mu, \nu)=\int \varphi(x)\,\d \mu(x)+\int \varphi^c(y)\,\d \nu(y).
\]

\begin{definition}[$c$-superdifferential]
Let $\varphi$ be a continuous function. The $c$-superdifferential $\partial^c \varphi \subset X \times X$ is defined as
\[
\partial^c \varphi:=\left\{ (x, y)\in X\times X: \varphi(x)+\varphi^c(y)=\frac{\dist^2(x, y)}2\right\}.
\]
The $c$-superdifferential $\partial^c \varphi(x)$ at $x\in X$ is the set of $y\in X$ such that $(x, y)\in \partial^c \varphi$.

\end{definition}

We have the following important theorem about the optimality of the transport plan and $c$-concave functions, see \cite[Theorem 1.13]{AG-U} for a proof.

\begin{theorem}[Fundamental theorem of optimal transport]\label{th:ftot}
Let $\Pi \in \mathcal{P}(X \times X)$ be a probability measure with $\mu$ and $\nu$ as the first and the second marginal,  such that $\int \dist^2(x,y)\,\d \Pi<+\infty$. Then the following are equivalent:
\begin{itemize}
\item  [(a)] The plan $\Pi$ is optimal, i.e. it realizes the minimum in the Kantorovich problem \eqref{eq:W2def}.
\item [(b)] There exists a $c$-concave function $\varphi$ such that $\max\{\varphi,0\}\in L^1(\mu)$ and $\supp(\Pi)\subset \partial^c \varphi$.
\end{itemize}
\end{theorem}

\subsubsection*{Curvature-dimension condition on metric measure spaces}

In order to formulate curvature-dimension conditions,  we recall the definition of the  distortion coefficients.  
 For $\kappa \in \R$, define the functions $s_\kappa, c_\kappa: [0, +\infty) \mapsto \R$ (on $[0, \pi/ \sqrt{\kappa})$ if $\kappa >0$) as:
\begin{equation} \label{eq:skappa}
s_\kappa(\theta):=\left\{\begin{array}{lll}
(1/\sqrt {\kappa}) \sin (\sqrt \kappa \theta), &\text{if}~~ \kappa>0,\\
\theta, &\text{if}~~\kappa=0,\\
(1/\sqrt {-\kappa}) \sinh (\sqrt {-\kappa} \theta), &\text{if} ~~\kappa<0
\end{array}
\right.
\end{equation}
and
\begin{equation} 
c_\kappa(\theta):=\left\{\begin{array}{lll}
\cos (\sqrt \kappa \theta), &\text{if}~~ \kappa>0,\\
1, &\text{if}~~\kappa=0,\\
 \cosh (\sqrt {-\kappa} \theta), &\text{if} ~~\kappa<0.
\end{array}
\right.
\end{equation}
It can be seen that $s'_\kappa=c_\kappa$,  and both functions $s_\kappa, c_\kappa$ are  solutions to the following  (Riccati-type) ordinary differential equation 
\begin{equation} \label{eq:riccati}
u''+\kappa  u=0.
\end{equation}
%For any $K\in \R$ and $N\in (1, +\infty]$, define the functions $S_{K, N}$ and $C_{K, N}$ on $[0, +\infty)$  (on $[0, \sqrt{\frac {N-1}K }\pi)$ if $K>0$ and $N<+\infty$) by
%\begin{equation} 
%S_{K, N}(\theta):=s_{\frac K{N-1}}(\theta)
%\end{equation}
%and
%\begin{equation} 
%C_{K, N}(\theta):=c_{\frac K{N-1}}(\theta)/s_{\frac K{N-1}}(\theta)=S'_{K, N}(\theta)/S_{K, N}(\theta).
%\end{equation}

For $K\in \R, N\in [1,\infty), \theta \in (0,\infty), t\in [0,1]$, 
we define the  distortion coefficients $\sigma_{K,N}^{(t)}$ and  $\tau_{K,N}^{(t)}(\theta)$ as
\begin{equation}\label{eq:Defsigma}
\sigma_{K,N}^{(t)}(\theta):= 
\begin{cases}
\infty, & \textrm{if}\ K\theta^{2} \geq N\pi^{2}, \crcr
t & \textrm{if}\ K \theta^{2}=0,  \crcr
\displaystyle   \frac{s_{\frac K{N}}(t\theta)}{s_{\frac K{N}}(\theta)} & \textrm{otherwise}
\end{cases}
\end{equation}
and
\begin{equation}\label{eq:deftau}
\tau_{K,N}^{(t)}(\theta): = t^{1/N} \sigma_{K,N-1}^{(t)}(\theta)^{(N-1)/N}.
\end{equation}

\bigskip

The following curvature-dimension conditions were introduced independently by Lott--Villani \cite{Lott-Villani09} and Sturm \cite{S-O1, S-O2} (with some differences).

\begin{definition}[$\cdkn$ condition]\label{def:CD}
Let $K \in \R$ and $N \in [1,\infty)$. A metric measure space  $(X,\dist,\mm)$ verifies $\cdkn$ if for any two $\mu_{0},\mu_{1} \in {\mathcal P}_{2}(X)$ 
with bounded support there exist a dynamical optimal plan $\pi \in \Opt(\mu_{0},\mu_{1})$ and an optimal transport plan $\Pi\in {\rm Opt}(\mu_0, \mu_1)$,  such that $\mu_{t}:=(e_{t})_{\sharp}\pi \ll \mm$ and for any $N'\geq N, t\in [0,1]$:
\begin{equation}\label{eq:defCD}
{\mathcal E}_{N'}(\mu_{t}) \geq \int_{X\times X} \tau_{K,N'}^{(1-t)} (\dist(x,y)) \rho_{0}^{-1/N'} (x)
+ \tau_{K,N'}^{(t)} (\dist(x,y)) \rho_{1}^{-1/N'} (y)\,\d\Pi(x, y)
\end{equation}
where the R\'enyi entropy  ${\mathcal E}_{N}$ is defined as
\begin{equation*}
{\mathcal E}_{N}(\mu):=
\left \{\begin{array}{ll}
\int \rho^{1-1/N} \,\d\mm &\text{if}~ \mu=\rho\,\mm\\
+\infty &\text{otherwise}.
\end{array}\right.
\end{equation*} 

\end{definition}

%\begin{definition}[$\cd$ condition]\label{def:cd}
 %We say that a metric measure space $\ms$  has  Ricci curvature lower bound $K$, or satisfies ${\rm CD}(K, \infty)$ condition,  if  the entropy  functional  $\ent{\mm}$ is  $K$-displacement convex  on  the $L^2$-Wasserstein space $(\mathcal{P}_2(X), W_2)$. This means,  for any two probability measures $\mu_0, \mu_1 \in \mathcal {P}_2 (X)$, there  is  an $L^2$-Wasserstein geodesic $(\mu_t)_{t\in [0,1]}$ such that  for any $t\in [0,1]$ it holds
 %\begin{equation}\label{eq1.5-intro}
%\frac{t(1-t)K}2 W^2_2(\mu_0, \mu_1)+{\rm Ent}_\mm(\mu_t) \leq t{\rm Ent}_\mm(\mu_1)+(1-t){\rm Ent}_\mm(\mu_0)
%\end{equation}
%where the relative entropy ${\rm Ent}_\mm$ is 

\medskip

\brk
It is worth recalling that if $(M,g)$ is a Riemannian manifold of dimension $n$ and 
$h \in C^{2}(M)$ with $h > 0$, then the  metric measure space  $(M,\dist_{g},h \, {\rm Vol}_{g})$ (where $\dist_{g}$ and ${\rm Vol}_{g}$ denote the Riemannian distance and volume induced by $g$) verifies $\cdkn$ with $N\geq n$ if and only if  (see  \cite[Theorem 1.7]{S-O2})
$$
 \Ric_{g,h,N} : =  \Ric_{g} - (N-n) \frac{\nabla_{g}^{2} h^{\frac{1}{N-n}}}{h^{\frac{1}{N-n}}} \geq  K g.  
$$
In particular if $N = n$ the generalized Ricci tensor $\Ric_{g,h,N}= \Ric_{g}$ makes sense only if $h$ is constant. 
\erk

\medskip

\brk \label{cdequiv}
A variant of the $\cdkn$ condition, called  reduced curvature dimension condition and denoted by  $\CD^{*}(K,N)$ \cite{BS-L},  asks for the same inequality \eqref{eq:defCD} of $\cdkn$ but  the
coefficients $\tau_{K,N}^{(t)}(\dist(\gamma_{0},\gamma_{1}))$ and $\tau_{K,N}^{(1-t)}(\dist(\gamma_{0},\gamma_{1}))$ 
are replaced by $\sigma_{K,N}^{(t)}(\dist(\gamma_{0},\gamma_{1}))$ and $\sigma_{K,N}^{(1-t)}(\dist(\gamma_{0},\gamma_{1}))$, respectively.
For both definitions there is a local version and it was recently proved in \cite{CavallettiEMilman-LocalToGlobal} that  on  an essentially  non-branching metric measure spaces with $\mm(X)<\infty$ (and in \cite{LZH22}  for general  $\sigma$-finite $\mm$), the  $\CD^{*}_{loc}(K,N)$, $\CD^{*}(K,N)$, $\CD_{loc}(K,N)$, $\cdkn$ conditions are all equivalent.
\erk

\subsection{Differential structure of metric measure spaces}

We recall some  facts about  calculus in metric measure spaces following the approach of \cite{AGS-C, AGS-M, G-O}. 

A function $f:X\to \R$ is called Lipschitz (or more precisely $L$-Lipschitz) if there exists a constant $L\geq 0$ such that 
$$
|f(x)-f(y)|\leq L \, \dist(x,y), \quad \forall x,y\in X.
$$
We denote by $\Lip(X, \dist)$ the space of real valued  Lipschitz functions on $(X,\dist)$ and with $\Lip_{c}(\Omega, \dist)\subset \Lip(X, \dist)$ the sub-space of Lipschitz functions  on $X$ with compact support contained in an open subset $\Omega\subset X$.

 Given $f\in \Lip(X, \dist)$, the \emph{local Lipschitz constant} 
$\lip f(x_{0})$ of $f$ at $x_{0}\in X$ is defined as
\begin{equation*}
\lip f(x_{0}):=\limsup_{x\to x_{0}}  \frac{|f(x)-f(x_{0})|}{\dist(x, x_{0})} \; \text{ if $x_{0}$ is not isolated}, \quad \lip f(x_{0})=0 \; \text{ otherwise}.
\end{equation*}

We say that $f\in L^2(X, \mm)$ is a Sobolev function in $W^{1,2}\ms$ if
\begin{eqnarray*}
 \inf \left\lbrace \liminf_{n \to \infty} \int_X \lip{f_n}^2\d\mm : f_n \in \Lip_{c}(X, \dist),\ \! f_n \to f \text{ in } L^2(X,\mm)\right\rbrace <+\infty.
\end{eqnarray*}
 For any $f\in W^{1,2}\ms$, there exists a sequence of Lipschitz  functions $(f_n) \subset L^2(X, \mm)$,  such that $f_n \to f$ and $\lip{f_n} \to G$ in $L^2$ for some $G \in L^2(X, \mm)$. There exists a minimal function $G$ in $\mm$-a.e. sense,  called  minimal weak  upper gradient (or weak gradient for simplicity) of  $f$, and we denote it by $|\D f|$.

\medskip

If $W^{1,2}\ms$ is a Hilbert space,   $(X,\dist,\mm)$ is called infinitesimally Hilbertian (cf. \cite{AGS-M, G-O}). In this case, for $f, u\in  W^{1,2}\ms$,  we define 
$$
\D f (\nabla u) : = \inf_{\epsilon > 0} \frac{ |\D (u + \epsilon f)|^{2} - |\D u|^{2} }{2\epsilon},
$$
and we have $\D f (\nabla u)= \D u (\nabla f)$. 

For infinitesimally Hilbertian metric measure spaces,  Cavalletti--E.Milman \cite{CavallettiEMilman-LocalToGlobal} and Li \cite{LZH22} prove the following equivalence.

\begin{proposition}[cf. Remark \ref{cdequiv}]\label{rcd}
For infinitesimally Hilbertian metric measure spaces,  
  the  $\CD^{*}_{loc}(K,N)$, $\CD^{*}(K,N)$, $\CD_{loc}(K,N)$, $\cdkn$ conditions are all equivalent,  and we denote them by $\rcdkn$. 
\end{proposition}

\begin{definition}[Measure valued Laplacian, cf. \cite{G-O}]\label{D:Laplace}
Let $\Omega\subset X$ be an open subset and let $u \in W^{1,2}\ms $. We say that $u$ is in the domain of the Laplacian of $\Omega$, and write $u \in \D({\bf \Delta},\Omega)$, provided  there exists a signed measure $\mu$ on $\Omega$ 
such that for any $f \in \Lip_{c}(\Omega, \dist)$ it holds 
\begin{equation}\label{eq:defTfLap}
 \int \D f (\nabla u) \,\d\mm = - \int f\,\d \mu.
\end{equation}

If  $\mu$ is unique,  we denote it by ${\bf \Delta} u$. If ${\bf \Delta} u \ll \mm$ and its density $\Delta u$ is locally finite,   we write $u\in {\rm D}({\Delta}, \Omega)$.
\end{definition}

\section{Alexandroff--Bakelman--Pucci estimate}\label{main}

\subsection{Contact sets}
\subsubsection*{$2$-contact set}\label{l2}

Let  $\Omega\subset X$ be a bounded  set,  $\De$ be a compact set,   $u$ be a continuous function and  $t>0$. 
Recall that the 2-contact set is defined as
  \begin{equation*}
 \Rem_2(\De, \Omega, u, t):=\left\{ x\in \overline \Omega:\, \exists y\in \De~\text {s.t.} ~\inf_{\overline \Omega}\Big(u+\frac{ \dist^2_y}{2t}\Big)=u(x)+\frac{ \dist^2(x, y)}{2t}\right\}.
 \end{equation*}
 
 \blm\label{lemma:gt2}
Let  $u \in C(\Omega)$. Then $-tu$ has a $c$-concave,  (upper) representative on $\Rem_2$, i.e. there is a $c$-concave function $\varphi$ so that   $\varphi=-tu$ on $\Rem_2$ and  $ \varphi\geq -tu$ on $\Omega$.

\elm
\begin{proof}
Define
\begin{equation*}
v(y):=
\left \{\begin{array}{ll}
\inf_{z\in \overline \Omega} \left( u(z)+\frac{ \dist^2(z,y)}{2t}\right) &\text{if}~ y\in \De\\
-\infty &\text{otherwise}.
\end{array}\right.
\end{equation*} 

By definition,
\begin{equation}\label{eq2:lemma2}
v(y)-u(z)\leq \frac{ \dist^2(z,y)}{2t}\qquad\forall (y, z)\in \De \times \overline \Omega
\end{equation}
and for any $x\in \Rem_2$ there is $y\in \De$ so that
\begin{equation}\label{eq1:lemma2}
v(y)-u(x)=\frac{\dist^2(x, y)}{2t}.
\end{equation}
So for any $x\in \Rem_2$, it holds
\begin{equation}\label{eq3:lemma1}
-tu(x)=\inf_{y\in \De}\left(-tv(y)+\frac{ \dist^2(x,y)}{2}\right)=\inf_{y\in X}\left(-tv(y)+\frac{ \dist^2(x,y)}{2}\right)=(tv)^c(x).
\end{equation}
For  $x\in \Omega \setminus \Rem_2$, by \eqref{eq2:lemma2} it holds 
\[
(tv)^c(x)=\inf_{y\in \De}\left(-tv(y)+\frac{ \dist^2(x,y)}{2}\right)\geq -tu(x).
\]
%Then there is $y_0\in \De$ and $z_0\in \overline \Omega$ so that 
%\begin{eqnarray*}
%(tv)^c(x)&=&-tv(y_0)+\frac{ \dist^2(x,y_0)}{2}=-t\Big(u(z_0)+\frac{ \dist^2(z_0,y_0)}{2t}\Big)+\frac{ \dist^2(x,y_0)}{2}\\
%&=&-tu(z_0)+\frac 12 \Big(\dist^2(z_0,y_0)-\dist^2(x,y_0)\Big)
%\end{eqnarray*}

%\[
%(tv)^c(x)=\inf_{y\in X}\left(-tv(y)+\frac{\dist^2(x,y)}{2}\right)=\inf_{y\notin \De}\left(-tv(y)+\frac{ \dist^2(x,y)}{2}\right)
%\]

Then we  define   $\varphi$ by
\begin{equation*}
\varphi(x):=
{(tv)^c(x) }
\end{equation*} 
which fulfils our request.
\end{proof}

\medskip

Let $\ms$ be an $\rcdkn$ space.  By  the Laplacian comparison theorem \cite[Theorem 5.14]{G-O}, we know that the $c$-concave function $\varphi$ obtained in the last lemma is in $\D({\bf \Delta},\Omega)$, and the positive part of ${\bf \Delta} \varphi$, denoted by $({\bf \Delta} \varphi)^+$,  is absolutely continuous and has bounded density. In addition, we have the following estimate concerning the negative part of ${\bf \Delta} \varphi$.
\begin{lemma}\label{lemma:comparison}
Let $\ms$ be an $\rcdkn$ space. Let ${\bf \Delta} \varphi={\bf \Delta} \varphi\llcorner_{\Rem_2}+{\bf \Delta} \varphi\llcorner_{\Omega\setminus \Rem_2}$ be a decomposition of ${\bf \Delta} \varphi$. Assume $u\in {\rm D}(\Delta, \Omega)$ with $\Delta u\in L^\infty$.  We have
\[
{\bf \Delta} \varphi\llcorner_{\Rem_2} \geq -\Delta u \,\mm\llcorner_{\Rem_2}.
\]
\end{lemma}
\begin{proof}
Let $(P_t \varphi)_{t\geq 0}$  be the heat flow from $\varphi$.  We claim that $\frac{P_t \varphi-\varphi}{t}\, \mm$ converge weakly to  ${\bf \Delta} \varphi$ as $t \downarrow 0$.
Given a non-negative  function $g\in {\rm TestF}:= \D({ \Delta},\Omega)\cap C_c(\Omega)$ with  $\Delta g\in L^\infty$. By \cite[Lemma 2.55]{mondino2023weak} we know 
\[
\lmt{t}{0} \frac{P_t g(x)-g(x)}{t}=\Delta g (x)~~~\text{for}~ \mm\text{-a.e.}~x\in \Omega.
\] By dominated convergence theorem we have
\begin{eqnarray*}
&&\lmt{t}{0}\int g \frac{P_t \varphi-\varphi}{t}\,\d \mm = \lmt{t}{0}\int \varphi \frac{P_t g-g}{t}\,\d \mm\\
&=& \int \varphi \Delta g\,\d \mm= \int g\,\d {\bf \Delta \varphi}.
\end{eqnarray*}
By Riesz--Markov--Kakutani representation theorem and
the density of  ${\rm TestF}$ in $\Lip_c(\Omega, \dist)$,   we know $ \frac{P_t \varphi-\varphi}{t}$ converge to ${\bf \Delta} \varphi$ as $t \downarrow 0$. 

Let  $\K\subset \Rem_2$ be a compact set and $h\in C(\K)$ be  such that $h>0$ on $\K$ and $h=0$ on $\Omega\setminus \K$. We can find a sequence $(h_n)_{n\in \N} \subset C_c(\Omega)$ such that $ h_n\geq h$ and
$h_n \downarrow h$ pointwisely. Then
\begin{eqnarray*}
&&\lmt{t}{0}\int h \frac{P_t \varphi-\varphi}{t}\,\d \mm \leq  \lmt{t}{0}\int h_n \frac{P_t \varphi-\varphi}{t}\,\d \mm\\
&=&  \int h_n\,\d {\bf \Delta \varphi}=  \int (h_n-h)\,\d {\bf \Delta \varphi}+ \int h\,\d {\bf \Delta \varphi}\\
&\leq& \int (h_n-h)\,\d ({\bf \Delta \varphi})^+ +\int h\,\d {\bf \Delta \varphi}.
\end{eqnarray*}
 Letting $n\to \infty$ we get
 \begin{equation}\label{eq:hflim}
 \lmt{t}{0}\int h \frac{P_t \varphi-\varphi}{t}\,\d \mm \leq \int h\,\d {\bf \Delta \varphi}.
 \end{equation}
 
  For any $x\in \Rem_2$,  by Lemma \ref{lemma:gt2} we know 
\[
P_t \varphi(x)-\varphi(x) \geq - t\big( P_t  u(x)- u(x)\big).
\]
Thus
\begin{eqnarray*}
\lmt{t}{0}\int h \frac{P_t \varphi-\varphi}{t}\,\d \mm \geq  -\lmt{t}{0}\int_\K h (x)\frac{P_t  u(x)- u(x)}{t}\,\d \mm(x)
= -\int h \Delta u \,\d \mm.
\end{eqnarray*}
Combining with \eqref{eq:hflim} and  the arbitrariness of  $h, \K$,  we know 
\[
{\bf \Delta} \varphi\llcorner_{\Rem_2} \geq -\Delta u \,\mm\llcorner_{\Rem_2}.
\]
\end{proof}
\subsubsection*{$1$-contact set}\label{l1}

\bigskip
Recall that the 1-contact set  $ \Rem_1(\De, \Omega,  u, t)$ is defined as
 \begin{equation*}
 \Rem_1(\De, \Omega,  u, t):=\left\{ x\in \overline \Omega:\, \exists y\in \De~\text {s.t.}  ~\dist(x, y)=t, \inf_{\overline \Omega}\big(u+\dist_y\big)=u(x)+\dist(x, y)\right\}
 \end{equation*}
 and
  \begin{eqnarray*}
  \Rem^*_1(\De, \Omega,  u)&:=&\bigcup_{t\geq 0} \Rem_1(\De, \Omega,  u, t)\\&=&\left\{ x\in \overline \Omega: \exists y\in \De~~\text {s.t.}~ \inf_{\overline \Omega}\big(u+\dist_y\big)=u(x)+\dist(x, y)\right\}.
 \end{eqnarray*}

We have the following  lemma concerning the Lipschitz regularity of $u$ on the contact set $\Rem^*_1$.  This lemma has its own interest in the  viewpoint of optimal transport theory.
 
\blm\label{lemma:gt}
Let  $\Omega\subset X$ be a bounded open set and $u$ be a continuous function defined on $\Omega$.  Define a 1-Lipschitz function $u^\dist$ on $\De$ by
\[
u^\dist(y):=\inf_{x\in \overline \Omega}\big(u(x)+\dist_y(x)\big)\qquad y\in \De.
\]Then  for any $t>0$ we have
\begin{itemize}
\item [$(1)$] $u$ is 1-Lipschitz on  $\Rem^*_1(\De, \Omega,  u)$.
\item [$(2)$] $u^\dist=u$ on $\De\cap \Rem^*_1(\De, \Omega,  u)$.
\item [$(3)$] For any $(x, z)\in \Rem^*_1(\De, \Omega,  u) \times \De$,  it holds
\[
-\dist(x, z)\leq  u^\dist(z)-u(x)\leq \dist(x, z).
\]
\end{itemize}

Furthermore,   for any $x\in \Rem^*_1(\De, \Omega,  u)$,  $y\in \De$ with  $u(x)+\dist(x, y)=u^\dist(y)$,   and any geodesic $\gamma\subset X$ connecting $x$ and $y$,  we have
 \[
 u(x')-u(x)=\dist(x', x)\qquad\forall x'\in \Rem^*_1(\De, \Omega,  u) \cap \gamma.
 \]
\elm
\begin{proof}
(1)  By definition, for any $x\in \Rem^*_1(\De, \Omega,  u)$ there is $y\in \De$ so that
\begin{equation}\label{eq0:lemma1}
u^\dist(y)=\inf_{\overline \Omega}\big(u+\dist_y\big)=u(x)+\dist_y(x).
\end{equation}
Then for any $x'\in \Rem^*_1(\De, \Omega,  u)$,
\begin{equation}\label{eq1:lemmagt}
u(x')+\dist_y(x')\geq u(x)+\dist_y(x).
\end{equation}
By triangle inequality
\[
u(x)-u(x')\leq \dist_y(x')-\dist_y(x) \leq \dist(x', x).
\]
By symmetry we also have
\[
u(x')-u(x)\leq  \dist(x', x)
\]
so   $u$ is 1-Lipschitz on $\Rem^*_1(\De, \Omega,  u)$.

\medskip

(2)  For $y\in \De\cap \Rem^*_1(\De, \Omega,  u)$,  on one hand, since $u$ is 1-Lipschitz on $\Rem^*_1(\De, \Omega,  u)$, we have 
\[
u(y)-u(x)\leq \dist(x, y)\qquad \forall x\in \Rem^*_1(\De, \Omega,  u)
\] 
so that
\[
u^\dist(y)=\inf_{\overline \Omega}\big(u+\dist_y\big)=\inf_{x\in \Rem^*_1}\big(u(x)+\dist(x, y)\big)\geq u(y).
\]
On the other hand, we have the following trivial inequality
\[
u^\dist(y)=\inf_{\overline \Omega}\big(u+\dist_y\big)\leq u(y)+\dist_y(y)=u(y).
\]
In conclusion, $u^\dist=u$ on $\De\cap \Rem^*_1(\De, \Omega,  u)$.
\medskip

(3) By definition of $u^\dist$,  we have
\begin{equation}\label{eq:udd}
u^\dist(z)-u(x)\leq \dist(x, z)\qquad (x, z)\in \overline\Omega \times \De,
\end{equation}
and for any $x\in \Rem^*_1(\De, \Omega,  u)$,   there is $y\in \De$ such that
\[
u^\dist(y)-u(x)=\dist(x, y).
\] 
So for any $(x, z)\in \Rem^*_1(\De, \Omega,  u) \times \De$,  we have
\[
-\dist(x, z)\leq u^\dist(z)+\dist(x, y)-u^\dist(y)=u^\dist(z)-u(x)\leq \dist(x, z)
\]
where in the first inequality we use the fact  that $u^\dist$ is 1-Lipschitz on $\De$.

\medskip

(4) At last, let $\gamma$ be a geodesic connecting $x\in \Rem^*_1(\De, \Omega,  u)$ and $y\in \De$.   By \eqref{eq1:lemmagt}, for any $x'\in  \Rem^*_1(\De, \Omega,  u) \cap \gamma$ it holds
\[
u(x')-u(x)\geq \dist_y(x)-\dist_y(x')= \dist(x, x').
\]
By (1),  $u$ is 1-Lipschitz on $\Rem^*_1(\De, \Omega,  u)$,  so we have $u(x')-u(x)= \dist(x, x')$.

\end{proof}

\bigskip

\begin{lemma}\label{lemma:gt3}
The function $-tu$ has a $c$-concave (upper) representative $\varphi$  on $\Rem_1(\De, \Omega,  u, t)$.  If  $u\in {\rm D}(\Delta, \Omega)$ with $\Delta u\in L^\infty$, then it holds a Laplacian estimate 
\[
{\bf \Delta} \varphi\llcorner_{\Rem_1} \geq -\Delta u \,\mm\llcorner_{\Rem_1}.
\] 
\end{lemma}
\begin{proof}

Define
\begin{equation*}
v(y):=
\left \{\begin{array}{ll}
\inf_{z\in \overline \Omega} \left( u(z)+\frac{ \dist^2(z,y)}{2t}\right) &\text{if}~ y\in \De\\
-\infty &\text{otherwise}.
\end{array}\right.
\end{equation*} 

By definition,
\begin{equation}\label{eq03:lemma2}
v(y)-u(z)\leq \frac{ \dist^2(z,y)}{2t}\qquad\forall (y, z)\in \De \times \overline \Omega.
\end{equation}
Given $x\in \Rem_1(\De, \Omega,  u, t)$,  there is $y_x\in \De$ so that
\begin{equation}\label{eq00:th1}
 u^\dist(y_x)= u(x)+\dist(x, y_x)=u(x)+t.
\end{equation}
Then 
\[
u(x)+\frac{ \dist^2(x,y_x)}{2t}= u^\dist(y_x)-t+\frac t2= u^\dist(y_x)-\frac t2.
\]
For any other $z\in \overline \Omega$, by Cauchy inequality and \eqref{eq:udd},  it holds
\begin{eqnarray*}
&& u(z)+\frac{ \dist^2(z,y_x)}{2t}\\
&\overset{\text{Cauchy}}\geq &u(z)-\frac t2+\dist(z, y_x) \overset{\eqref{eq:udd}} \geq  u^{\dist}(y_x)-\frac t2=u(x)+\frac{ \dist^2(x,y_x)}{2t},
\end{eqnarray*}
so
\begin{equation}\label{eq02:lemma2}
u(x)+\frac{\dist^2(x, y_x)}{2t}=\inf_{z\in \overline \Omega} \left( u(z)+\frac{ \dist^2(z,y_x)}{2t}\right)=v(y_x).
\end{equation}

Combining \eqref{eq03:lemma2} and \eqref{eq02:lemma2} we get
\begin{equation}\label{eq01:lemma2}
-tu(x)=\inf_{y\in \De}\left(-tv(y)+\frac{ \dist^2(x,y)}{2}\right)=\inf_{y\in X}\left(-tv(y)+\frac{ \dist^2(x,y)}{2}\right)=(tv)^c(x)
\end{equation}
for  $x\in \Rem_1$ and $-tu(x) \leq  (tv)^c(x)$ for $x\in \overline \Omega$.

Thus $(tv)^c$ is the desired upper representative, and the Laplacian estimate can be proved using the same argument as Lemma \ref{lemma:comparison}.
\end{proof}

\subsection{Functional ABP estimate}
To study the   curvature-dimension condition with  finite dimension parameter (cf. Definition \ref{def:CD}),
Erbar--Kuwada--Sturm \cite{EKS-O} introduced  a notion called \emph{entropic curvature-dimension condition} $\mathrm {CD}^e(K, N)$.  For  infinitesimally Hilbertian metric measure spaces,  they show that $\mathrm {CD}^e(K, N)$ is equivalent to  the reduced curvature-dimension condition ${\rm CD}^*(K, N)$ introduced by Bacher--Sturm \cite{BS-L} (cf. Remark \ref{cdequiv} and Proposition \ref{rcd}). Following their footprints, we can prove the following differential inequality (see \cite[Lemma 2.2]{EKS-O}).

\begin{lemma}[$(K, N)$-convexity]\label{lemmakn}
Let $\ent{\mm}$ be the relative entropy defined as
\begin{equation*}
{\rm Ent}_\mm(\mu):=
\left \{\begin{array}{ll}
\int \rho\ln \rho\,\d \mm &\text{if}~ \mu=\rho\,\mm\\
+\infty &\text{otherwise}.
\end{array}\right.
\end{equation*} 
 and $U_N:=\exp{\left(-\frac 1N \ent{\mm}\right)}$.  Assume that $\ms$ is an  ${\rm RCD}(K, N)$ space.
Then for each pair $\mu_0, \mu_1 \in \mathcal P_2(X)$ with $\mu_0, \mu_1 \ll \mm$, there is a   geodesic $(\mu_t)_{t \in [0,1]}$ in $\big(\mathcal P_2(X), W_2\big)$ from $\mu_0$ to $\mu_1$ so that for all $t\in [0, 1]$ we have 
\begin{equation}\label{1:lm3}
U_N(\mu_1)\leq c_{K/N}\big(W_2(\mu_0, \mu_1)\big) U_N(\mu_0)+\frac{s_{K/N}\big(W_2(\mu_0, \mu_1)\big)}{W_2(\mu_0, \mu_1)}\frac {\d^-}{\d t} \restr{t=0} U_N(\mu_t)
\end{equation}
where
\[
\frac {\d^-}{\d t} \restr{t=0} U_N(\mu_t):=\liminf_{h\downarrow 0} \frac{U_N(\mu_h)- U_N(\mu_0)}{h}.
\]
\end{lemma}
\begin{proof}
By \cite[Definition 3.1, Definition 2.7]{EKS-O}, 
 there is a  constant speed geodesic $(\mu_t)_{t \in [0,1]}$ in the Wasserstein space $\big(\mathcal P_2(X), W_2\big)$ connecting $\mu_0$ and $\mu_1$, so that for all $t\in [0, 1]$ it holds
\begin{equation}\label{0:lm3}
 U_N(\mu_t)\geq \sigma^{(1-t)}_{K/N}\big(W_2(\mu_0, \mu_1)\big) U_N(\mu_0)+\sigma^{(t)}_{K/N}\big(W_2(\mu_0, \mu_1)\big) U_N(\mu_1).
\end{equation}

Subtracting $U_N (\mu_0)$ on both sides of \eqref{0:lm3},  dividing by $t$ and letting $t \downarrow 0$, we get \eqref{1:lm3}.
\end{proof}

\medskip

Next,  we  need to find an upper bound  of $\frac {\d^-}{\d t} \restr{t=0} U_N(\mu_t)$.  To do this, we  make use of a strategy of Gigli \cite[Proposition 5.10]{G-O}.

\begin{lemma}[Bound from above on the derivative of the entropy]\label{lemmadr}
Let $\ms$ be an   $\rcdkn$ metric measure space.
Let $ \mu_0\in \mathcal P_2(\overline \Omega)$ be such that $\mu_0\ll \mm\llcorner_{\Omega}$ with $\mu_0 = \rho \mm\llcorner_{\Omega}$. Assume that $\mm(\partial \Omega) = 0$ and assume also that the restriction of $\rho$ to $\Omega$ is Lipschitz and bounded from below by a positive constant. Let $\mu_1 \in  \mathcal P_2(X)$, $\pi \in  \Opt(\mu_0,\mu_1)$ and $\mu_t:=(e_t)_\sharp \pi$. Then it holds
\begin{equation}\label{1:lm4}
\liminf_{t\downarrow 0} \frac{U_N(\mu_t)- U_N(\mu_0)}{t}\leq  \frac 1N U_N(\mu_0)  \int_\Omega \D \rho(\nabla \varphi)\,\d \mm
\end{equation}
where   $\varphi$ is any Kantorovich potential from $\mu_0$ to $\mu_1$.
\end{lemma}
\begin{proof}

We work on the space $(\overline \Omega, \dist, \mm \llcorner_{\Omega})$.      For any $\nu\in \mathcal P(\overline \Omega)$ with $\nu\ll \mm\llcorner_{\Omega}$, the concavity of the function  $e^{-x}$ gives
\[
U_N(\nu)- U_N(\mu_0)\leq  -U_N(\mu_0)\left(\frac 1N \ent{\mm}(\nu)-\frac 1N \ent{\mm}(\mu_0)\right).
\]
Since $\rho, \rho^{-1}$ are positive and bounded,  the function $\ln \rho: \overline \Omega \to \R$ is bounded. Thus
the convexity of $x\ln x$ gives
\[
\ent{\mm}(\nu)- \ent{\mm}(\mu_0) \geq \int_\Omega \ln (\rho)\Big(\frac{\d \nu}{\d \mm}-\rho\Big)\,\d \mm, \qquad \forall \nu \in \mathcal P(\overline \Omega), \nu\ll\mm.
\]
Then 
\[
U_N(\nu)- U_N(\mu_0)\leq  -\frac 1N U_N(\mu_0) \int_\Omega \ln (\rho)\Big(\frac{\d \nu}{\d \mm}-\rho\Big)\,\d \mm.
\]
Plugging $\nu:=(e_t)_\sharp \pi$, dividing by $t$ and letting $t \downarrow 0$ we get
\[
\lmti{t}{0} \frac{U_N(\nu)- U_N(\mu_0)}t \leq -\frac 1N U_N(\mu_0)  \lmti{t}{0}  \int_\Omega  \frac {\ln \rho \circ e_t-\ln\rho \circ e_0}t\,\d \pi.
\]

 Applying  \cite[Proposition 5.9]{G-O},  we get
\[
  \lmti{t}{0}  \int_\Omega  \frac {\ln \rho \circ e_t-\ln\rho \circ e_0}t\,\d \pi \geq -\int_\Omega \D \rho(\nabla \varphi)\,\d \mm.
\]
Combining the estimates above we complete the proof.

\end{proof}
\medskip

Next  we will prove a  \emph{functional version} of the ABP estimate.

\begin{theorem}[A functional ABP estimate]\label{lemmaEVI}
Let $\ms$ be an    $\rcdkn$ metric measure space. Assume that $\mm(\partial \Omega) = 0$.
Let $\mu_0\in \mathcal P_2(\overline \Omega)$ be  with non-negative density $\rho\in \Lip(\Omega, \dist)$, $\mu_1\in \mathcal P_2(X)$ be  with bounded support. Then it holds
\begin{equation}\label{1:lm5}
U_N(\mu_1)\leq \left(c_{K/N}\big(W_2(\mu_0, \mu_1)\big) -\frac{s_{K/N}\big(W_2(\mu_0, \mu_1)\big)}{NW_2(\mu_0, \mu_1)}\left( \int_\Omega  \rho\,\d {\bf \Delta} \varphi\right)\right)U_N(\mu_0)  
\end{equation}
where   $\varphi$ is any Kantorovich potential from $\mu_0$ to $\mu_1$.
\end{theorem}
\begin{proof}
 {\bf Step 1.} Exponential map:

   Let $\varphi$ be  a Kantorovich potential from $\mu_0$ to $\mu_1$,  which is a $c$-concave function. Without loss of generality, we may assume that $\varphi$ is a real-valued function on $\Omega$.
  Consider the multivalued map $  \Omega\ni x\mapsto T_\varphi(x)\subset \pr(X)$ defined by: $$\nu\in T_\varphi(x) ~~\text{if and only if}~~\supp \nu\subset\partial^c  \varphi(x).$$ By  measurable selection theorem (see e.g. \cite[Theorem 6.9.3]{B-M} and  \cite[proof of THEOREM 5.14, page 71]{G-O}), there exists a measurable map $x\mapsto \eta_x$ such that $\eta_x\in T_\varphi(x)$ for any $x\in   \Omega$.
  
   For any $\mu\in \mathcal P(X)$ with $\mu\llcorner_\Omega \ll \mm\llcorner_\Omega$,   define a probability measure $T_\varphi(\mu)$ by 
   \begin{equation}\label{eq1:br}
      T_\varphi(\mu):=\int \eta_x\,\d \mu(x).
   \end{equation}
   By the fundamental theorem of optimal transport (cf. Theorem \ref{th:ftot}), $ \varphi$ is a Kantorovich potential from $\mu$ to $T_\varphi(\mu)$. By \cite{RS-N} (see also \cite{GRS-O}) the Monge's problem is uniquely solvable on $\rcdkn$ spaces,  and up to  an additive constant,  $ \varphi$ is the unique Kantorovich potential from $\mu$ to $T_\varphi(\mu)$. In this case,   for $\mm$-a.e. $x\in \Omega$, there  is a $y_x\in \partial^c \varphi(x)$.  
Then we can define a map $\nabla \varphi:\Omega \to \partial^c \varphi$   by $\nabla \varphi(x)=y_x$ such that 
   \begin{equation}\label{eq2:br}
T_\varphi(\mu)=(\nabla \varphi)_\sharp \mu~~~\forall \mu\ll \mm~\text{on}~\Omega.
   \end{equation}

\medskip

 {\bf Step 2.} Entropy estimate:
 
Firstly,  let  $(\zeta_m)_{m\in \N} \subset \Lip_c(\Omega, \dist)$ be a sequence of non-negative functions,  such that $\zeta_m \uparrow \rho$ pointwisely.
Denote $$\nu_{m, n, 0}=C_{m,n} \left(\zeta_m+\frac 1n\right)\,\mm\llcorner_\Omega,~~~\zeta_{m,n, 0}=C_{m,n} \left(\zeta_m+\frac 1n\right)$$
where $C_{m,n}$ are normalizing constants such that $\nu_{m,n, 0}\in \mathcal P(\Omega)$. By Step 1, there is a (unique) probability measure  $\nu_{m, n, 1}=T_\varphi(\nu_{m,n, 0})$ so that  $\varphi$ is a Kantorovich potential from $\nu_{m,n, 0}$ to $\nu_{m, n,1}$. By Lemma \ref{lemmakn} and Lemma \ref{lemmadr} we get 
\begin{equation}\label{2:lm5}
\frac{U_N(\nu_{m,n, 1})}{U_N(\nu_{m,n,0})}\leq c_{K/N}\big(W_2(\nu_{m, n,0}, \nu_{m,n,1})\big) +\frac{s_{K/N}\big(W_2(\nu_{m,n,0}, \nu_{m,n,1})\big)}{NW_2(\nu_{m,n,0}, \nu_{m,n,1})}\left( \int_\Omega \D \zeta_{m,n, 0}(\nabla \varphi)\,\d \mm\right).
\end{equation}

Denote $\nu_{m,  0}=C_{m} \zeta_m \mm,  \nu_{m, 1}=T_\varphi(\nu_{m,  0}) \in \mathcal P(\Omega)$.
By \eqref{eq1:br}, \eqref{eq2:br} and monotone convergence theorem, we can see that $\nu_{m,n, i} \overset{W_2}\to \nu_{m, i}$ and $U_N(\nu_{m, n, i})  \to U_N(\nu_{m, i}) $ for $i=1,2$,  as $n\to \infty$.
By locality of the weak gradient, we also have
\begin{equation}\label{eq:ibp}
\lmt{n}{\infty}\int_\Omega \D \zeta_{m,n, 0}(\nabla \varphi)\,\d \mm= C_m\int_\Omega \D \zeta_m(\nabla \varphi)\,\d \mm=-C_m\int_\Omega  \zeta_m\,\d {\bf \Delta} \varphi.
\end{equation}

Letting $n\to \infty$ in \eqref{2:lm5} and combining with \eqref{eq:ibp} we get
\begin{equation}\label{eq3:br}
\frac{U_N(\nu_{m,1})}{U_N(\nu_{m,0}) }\leq c_{K/N}\big(W_2(\nu_{m,0}, \nu_{m,1})\big) -C_m\frac{s_{K/N}\big(W_2(\nu_{m,0}, \nu_{m,1})\big)}{NW_2(\nu_{m,0}, \nu_{m,1})}\left( \int_\Omega  \zeta_m\,\d {\bf \Delta} \varphi\right)
\end{equation}

By \eqref{eq1:br} and \eqref{eq2:br} again, we can see that $\nu_{m, i} \overset{W_2}\to \mu_{i}$ and $U_N(\nu_{m,  i})  \to U_N(\mu_{i}) $ for $i=1,2$,  as $m\to \infty$.
Letting $m\to \infty$ in \eqref{eq3:br} and noticing that $\lmt{m}{\infty} C_m =1$,  we get
 \eqref{1:lm5}.
\end{proof}

\subsection{Main results}

\subsubsection*{Proof of the main theorem}

\begin{proof}[Proof of Theorem \ref{th2}]

The proof is divided  into four steps. Without loss of generality, we  assume that  $\mm(\De)>0$.
 
\medskip

 {\bf Step 1.}

For any  $c$-concave function $\phi$, from the proof of Theorem \ref{lemmaEVI},   for $\mm$-a.e. $x\in \Omega$, there  is a unique geodesic $\gamma^x$ such that $\gamma^x_0=x$, $\gamma^x_1\in \partial^c \phi(x)$.  
Then we can define a map $\nabla \phi:\Omega \to X$   by $\nabla \phi(x)=\gamma^x_1$ such that  for any $\mu\ll \mm$,  $ \phi$ is a Kantorovich potential from $\mu$ to
$
T_\phi(\mu):=(\nabla \phi)_\sharp \mu$.
 In addition,  we define maps 
$\nabla \phi_t(x):=\gamma^x_t, t\in (0, 1)$ such that $((\nabla \phi_t)_\sharp \mu)_{t\in [0,1]}$ is the unique geodesic from $\mu$ to  $T_\phi(\mu)$ in the Wasserstein space.

\medskip

{\bf Step 2.}

 Let   $\varphi$ be  a $c$-concave representative of $-tu$ given in Lemma \ref{lemma:gt2} (or Lemma \ref{lemma:gt3} respectively).
By Lemma \ref{lemma:gt2} and the compactness of $\De$, we have
  \begin{equation}\label{eq1:th2}
  \De \subset \Big\{y\in \partial^c \varphi(x): x\in \Rem_2 \Big\}~\text{and}~\Rem_2 \subset \Big\{x\in \partial^c \varphi^c(y): y\in \De \Big\}. 
\end{equation} 
By Lemma \ref{lemma:gt}, Lemma \ref{lemma:gt3} (and the assumption of the theorem), we also have
  \begin{equation}\label{eq1.1:th2}
  \De \subset \Big\{y\in \partial^c \varphi(x): x\in \Rem_1 \Big\}~\text{and}~\Rem_1 \subset \Big\{x\in \partial^c \varphi^c(y): y\in \De \Big\}. 
\end{equation}

Let   $\nu:=\frac 1{\mm(\De)}\mm\llcorner_{\De}\in \mathcal P_2(X)$.  By Step 1,  there is a (unique) measure $\mu\in \pr(\Rem_i)$ such that: $\varphi^c$ is a Kantorovich potential from $\nu$ to $\mu$, and $\varphi$ is a Kantorovich potential from $\mu$ to $\nu$.   By \cite[Theorem 4.3, Step 4]{mondino2022lipschitz},   $\mu\ll\mm$ and has  bounded density $\rho$.   So we have $\nu=T_\varphi (\mu)$. 
   
    \medskip
 
   Denote by ${\rm M}_{ \Omega}\subset \pr(\Omega)$ the space of all probability measures  on  $ \Omega$  with Lipschitz density  bounded from below by a positive constant.

{\bf Claim:} Let  ${\K}\subset \Rem_i \subset \Omega$  be a compact set such that $\rho$ is bounded from  below by a positive constant on ${\K}$.  Let $c_{\K}$ be the normalizing constant so that $\mu_{\K}:=c_{\K} \mu\llcorner_{\K}\in \pr(\Omega)$.  We can find a sequence $(\mu_{n})_{n\in \N}\subset {\rm M}_{\Omega}$   such  that
\begin{itemize}
\item [(1)]  $\mu_n=\rho_n \,\mm$ with uniformly bounded $\rho_n\in \Lip_c(\Omega, \dist)$;
\item [(2)]  $\lmt{n}{\infty} \mu_n ( {\K})= 1$ and 
$$\lmt{n}{\infty}W_2(\mu_{\K}, \mu_{n})=\lmt{n}{\infty} {\rm Ent}_{\mu_{\K}}(\mu_n\llcorner_{{\K}})=0.$$
\item  [(3)] for  $\nu={\bf \Delta} \varphi\llcorner_{\Omega\setminus \Rem_i}$ it holds 
\begin{equation}\label{eq:conv}
\lmt{n}{\infty}\int\rho_n\,\d \nu=0
\end{equation}
\end{itemize}
 
\textit{Proof of the claim:}

Let $(H_t \mu_{\K})_{t>0}$ be  the gradient flow of the relative entropy ${\rm Ent}_{\mm}$  from $\mu_{\K}$,   in the $2$-Wasserstein space. 
It is known that $H_t \mu_{\K}\ll \mm$ for  $t>0$ and they have continuous, uniformly  bounded densities.  Denote $\mu_{\K}=\rho_{\K}\,\mm=c_{\K}\rho\chi_{\K}\,\mm$.  By \cite{AGS-C} we also know the  heat flow $P_t \rho_\K$ from $\rho_{\K}$ coincides with the density of $H_t \mu_{\K}$,  so  we can write $H_t \mu_{\K}=P_t \rho_{\K} \,\mm$. By \cite[PROPOSITION 6.4]{AGS-M} we also know 
$P_t \rho_\K$ is Lipschitz.

Furthermore, it is known that $\lmt{t}{0} W_2(H_t \mu_{\K}, \mu_{\K})=0$,  $\lmt{t}{0} {\rm Ent}_{\mm}(H_t \mu_{\K})={\rm Ent}_{\mm}( \mu_{\K})$, and $P_t \chi_{{\K}}$ converge to $\chi_{{\K}}$ in $L^2$.  So for  $\tilde \mm:=\frac 1{\mm({\K})} \mm\llcorner_{{\K}}$, we have
\begin{eqnarray*}
\lmt{t}{0}\int \chi_{\K}\,\d H_t \tilde \mm=\lmt{t}{0}\int P_t\chi_{\K}\,\d \tilde \mm=
  \int  \chi_{{\K}} \,\d \tilde \mm=1.
\end{eqnarray*}
Thus $\lmt{t}{0} H_t \tilde \mm({\K})= 1$ and $\lmt{t}{0} H_t \tilde \mm({X\setminus \K})= 0$. Recall that $\rho_{\K}$ is bounded, by maximal principle of the heat flow we have 
\begin{equation}\label{eq:error}
\lmt{t}{0} H_t \mu_{\K}(X \setminus {\K})= 0~~\text{and}~~\lmt{t}{0} H_t \mu_{\K}( {\K})= 1.
\end{equation}

By direct computation we have
 \begin{eqnarray*}
 &&{\rm Ent}_{\mu_{\K}} \big((H_t\mu_{\K})\llcorner_{{\K}}\big)\\&=&\int_{{\K}} \ln (P_t \rho_{\K}/\rho_{\K})\,\d H_t \mu_{\K}\\
 &=&\int_{{\K}} \ln (P_t \rho_{\K})\,\d H_t \mu_{\K}-\int_{\K} \ln \rho_{\K}\,\d H_t\mu_{\K}\\
 &=&\underbrace{\int_{X} \ln (P_t \rho_{\K})\,\d H_t \mu_{\K}}_{{\rm Ent}_{\mm} (H_t\mu_{\K})}-\int_{X\setminus {\K}} \ln (P_t \rho_{\K})\,\d H_t \mu_{\K}-\int_{\K} \ln \rho_{\K}\,\d H_t\mu_{\K}.
 \end{eqnarray*}
 By \eqref{eq:error} and Jensen's inequality we know 
 $$\lmt{t}{0}\int_{X\setminus {\K}} \ln (P_t \rho_{\K})\,\d H_t \mu_{\K}=0.$$
  By choice of $\K$,  $\ln \rho_{\K}$ is bounded.  Then we have 
   $$\lmt{t}{0}\int_{\K} \ln \rho_{\K}\,\d H_t\mu_{\K}=\lmt{t}{0}\int_{\K} P_t \rho_{\K} \ln \rho_{\K} \,\d \mm={\rm Ent}_{\mm} (\mu_{\K}).$$
Combining with the continuity of the  entropy along the heat flow,    we get $$\lmt{t}{0}{\rm Ent}_{\mu_{\K}} \big((H_t\mu_{\K})\llcorner_{{\K}}\big)=0.$$

At last, for any $n\in \N$,  there is a compact set $E_n\subset \Omega\setminus \Rem_i$ so that 
\begin{equation}\label{eq1:err}
\|\rho_\K \|_{L^\infty}\nu (\Omega \setminus E_n)<\frac 1{2n}.
\end{equation} 
Denote $\nu_n:=\nu\llcorner_{E_n}/ \nu(E_n)$.  Note that  $\lmt{t}{0} W_2(H_t \nu_n, \nu_n)=0$ and $\dist(\K, E_n)>0$, it holds
$$
\lmt{t}{0} \int_{E_n}  P_t \rho_{\K}\,\d \nu=\lmt{t}{0} \nu(E_n) \int  \rho_{\K}\,\d H_t \nu_n=0.
$$
So there is $t_n>0$ so that 
\begin{equation}\label{eq2:err}
\int_{E_n}  P_{t_n} \rho_{\K}\,\d \nu<\frac1{2n}.
\end{equation}
Combining  \eqref{eq1:err} and  \eqref{eq2:err} we get
\begin{eqnarray*}
\left|\int P_{t_n} \rho_{\K}\,\d \nu \right| &=& \left| \int_{E_n} P_{t_n} \rho_{\K}\,\d \nu+\int_{\Omega \setminus {E_n}} P_{t_n} \rho_{\K}\,\d \nu \right|\\
& \leq & \frac 1{2n}+ \frac 1{2n}=\frac 1n.
\end{eqnarray*}

For any $n\in \N$,  we define $$\mu_n= a_n\left({(H_{t_n}\mu_{\K})\llcorner_{\Omega}}+\frac 1n\right)$$  where $a_n$ is the normalizing constant.
From the construction above we can see that $\mu_n$ fulfils our request.

 \medskip

{\bf Step 3.}
Given a compact set ${\K}\subset \Rem_i$ and a sequence $(\mu_n)_{n\in \N}\subset {\rm M}_{ \Omega}$  constructed in the last step.
For any $t\in (0,1)$, denote $\mu_{n, t}=(\nabla \varphi_t)_\sharp \mu_n$. 
By local compactness of $\ms$, we may assume that $\mu_{n, t} \weakto  \mu_t$ for some $ \mu_t\in \pr (\Omega)$ as $n\to \infty$. By stability of optimal transport (cf. \cite[Theorem 5.20]{V-O}) and the uniqueness of optimal transport map again, we know  $ \mu_t=(\nabla \varphi_t)_\sharp \mu_{\K}$ is the unique $t$-intermediate point between $\mu_{\K}$ and a uniform distribution $ T_\varphi(\mu_{\K})=\frac 1{\mm(\De_{\K})}\mm\llcorner_{\De_{\K}}$ on some measurable set $\De_{\K}\subset \De$.    
By $\rcdkn$ condition \cite{R-I} and  \cite[Theorem 4.3]{mondino2022lipschitz}, we also know $\mu_{n, t}, \mu_t\ll \mm$,  and  for any  $t>0$, $ (\mu_{n,t})_{n\in \N}$ have  uniformly  bounded densities $\rho_{n, t}$.

 Furthermore, the density of $\mu_t$ satisfies
\begin{equation}\label{eq0:ent}
\|\rho_t \|_{L^\infty} \leq \frac 1{\mm(\De_{\K})}+o(1),~~\text{as}~t\to 0.
\end{equation}
Given $\epsilon >0$,  by \eqref{eq0:ent} we have
\begin{equation}\label{eq1:ent}
\|\rho_t \|_{L^\infty} \leq \frac 1{\mm(\De_{\K})}+\epsilon
\end{equation}
for $t\in (0,1)$  close enough to 1.

Next we  will estimate ${\rm Ent}_{\mm}(\mu_{n, t})$.   Denote $\mu_{t}=\rho_{t}\,\mm,  \mu_{n, t}=\rho_{n, t}\,\mm$,  and $\tilde\mu_{n}=c_n \mu_{n}\llcorner{{\K}} \in \mathcal P(\Omega)$ for some  normalizing constants $c_n$, $n\in \N$.   We have
\begin{eqnarray*}
&&{\rm Ent}_{\mm}(\mu_{n, t})\\&=&\int_{\nabla \varphi_t(\Omega)} \ln \rho_{n, t}\,\d \mu_{n, t}\\
&=&\int_{\nabla \varphi_t({\K})} \ln \rho_{n, t}\,\d \mu_{n, t}+\int_{\nabla \varphi_t(\Omega\setminus {\K})} \ln \rho_{n, t}\,\d \mu_{n, t}
\end{eqnarray*}
and
\begin{eqnarray*}
&&{\rm Ent}_{(\nabla \varphi_t)_\sharp( \mu_{\K})}\big((\nabla \varphi_t)_\sharp(\tilde \mu_n)\big)\\
&=&\int_{\nabla \varphi_t({\K})} \ln (c_n\rho_{n, t})\,\d (c_n\mu_{n, t})-\int_{\nabla \varphi_t({\K})} \ln \rho_{t}\,\d (c_n\mu_{n, t})\\
&=&{c_n\int_{\nabla \varphi_t({\K})} \ln \rho_{n, t}\,\d \mu_{n, t}}+\underbrace{\int_{\nabla \varphi_t({\K})} c_n\ln c_n\,\d \mu_{n, t}}_{=\ln c_n}-\int_{\nabla \varphi_t({\K})} \ln \rho_{t}\,\d (c_n\mu_{n, t}).
\end{eqnarray*}
Combining  with \eqref{eq1:ent},  we get
\begin{eqnarray*}
&&{\rm Ent}_{\mm}(\mu_{n, t})\\&=&\frac 1{c_n} {\rm Ent}_{(\nabla \varphi_t)_\sharp( \mu_\K)}\big((\nabla \varphi_t)_\sharp(\tilde \mu_n)\big)-\frac{\ln c_n}{c_n}+\frac 1{c_n} \int_{\nabla \varphi_t({\K})} \ln \rho_{t}\,\d (c_n\mu_{n, t})\\&&+\int_{\nabla \varphi_t(\Omega\setminus {\K})} \ln \rho_{n, t}\,\d \mu_{n, t}\\
&\leq&  \frac 1{c_n} {\rm Ent}_{(\nabla \varphi_t)_\sharp( \mu)}\big((\nabla \varphi_t)_\sharp(\tilde \mu_n)\big)-\frac{\ln c_n}{c_n}+\frac 1{c_n} \ln \big(1/\mm(\De_{\K})+\epsilon\big)\\&&+ \ln \|\rho_{n, t}\|_{L^\infty} \mu_{n}(\Omega\setminus \Rem_i).
\end{eqnarray*}

By \cite[Lemma 9.4.5]{AGS-G} we have
\begin{equation}\label{eq:reverseent}
0\leq {\rm Ent}_{(\nabla \varphi_t)_\sharp( \mu_\K)}((\nabla \varphi_t)_\sharp(\tilde \mu_n))\leq {\rm Ent}_{\mu_\K}(\tilde\mu_{n}).
\end{equation}

Combining these estimates above, the properties (1) and (2) in {Step 2},  and  $\lmt{n}{\infty}c_n=1$,  we obtain
\begin{equation}\label{eq:entconv}
\lmt{n}{\infty}{\rm Ent}_{\mm}(\mu_{n, t}) \leq \ln \big(1/\mm(\De_{\K})+\epsilon\big)
\end{equation}
for $t\in (0,1)$  close enough to 1.

 \medskip
 
{\bf Step 4.}

Applying  Proposition \ref{lemmaEVI}  (with $\mu_1=\mu_{n, t},  \mu_0=\mu_{n}$ and a re-parametrization) we obtain
\begin{eqnarray*}
U_N(\mu_{n, t})&\leq& c_{K/N}\big(W_2(\mu_{n}, \mu_{n, t} )\big) U_N(\mu_{n})\\&&-t\frac{s_{K/N}\big(W_2(\mu_{n}, \mu_{n, t} )\big)}{W_2(\mu_{n}, \mu_{n, t} )}\left(\frac 1N U_N(\mu_{n})   \int \rho_n\,\d {\bf  \Delta \varphi }\right).
\end{eqnarray*}
Letting $n\to \infty$ and combining with  \eqref{eq:entconv},  upper semi-continuity of the functional $U_N(\cdot) $ (cf. \cite[Lemma 4.1]{S-O1}),  \eqref{eq:conv} and the Laplacian estimate in  Lemma \ref{lemma:comparison}, Lemma \ref{lemma:gt3},  we obtain
\begin{eqnarray*}
\frac{\big(1/\mm(\De_{\K})+\epsilon\big)^{-\frac 1N}}{U_N(\mu_{\K})}\leq c_{K/N}\big(W_2(\mu_{\K}, \mu_t)\big) +\frac{ts_{K/N}\big(W_2(\mu_{\K}, \mu_t)\big)}{NW_2(\mu_{\K}, \mu_t)}   \int  \rho_\K \Delta u \,\d \mm.
\end{eqnarray*}

By Jensen's inequality (cf. \cite[Lemma 4.1]{S-O1}), we get $$U_N(\mu_{\K}) \leq \mm(\K)^{\frac 1N}\leq \mm(\Rem_i)^{\frac 1N}.$$ Note  also that the function $\R^*\ni x\mapsto \frac {s_{K/N}(x)}x$ is non-increasing for $K\geq 0$ and increasing for $K<0$. Letting $t\to 1$, $\epsilon\to 0$ and $\mm(\Rem_i \setminus \K) \to 0$,  we get
\[
\mm(\De)\leq
\begin{cases}
\displaystyle \mm(\Rem_i)\left(c_{K/N}(\Theta)+\frac{ts_{K/N}(\Theta)}{N\Theta} \|(\Delta u)^+\|_{L^\infty(\overline \Omega)}\right)^N, & \textrm{if}\ K<0, \crcr
\displaystyle \mm(\Rem_i)\left(1+\frac tN \|(\Delta u)^+\|_{L^\infty(\overline \Omega)}\right)^N & \textrm{if}\ K =0,  \crcr
\displaystyle   \mm(\Rem_i)\left(c_{K/N}(\Phi)+\frac{ts_{K/N}(\Phi)}{N\Phi}  \|(\Delta u)^+\|_{L^\infty(\overline \Omega)}\right)^N & \textrm{if}\ K>0,
\end{cases}
\]
where $(\Delta u)^+$ denotes the positive part of $\Delta u$,  $\Theta:=\sup_{(x, y)\in \De \times \Omega} {\dist(x, y)}$ and $\Phi:=\inf_{(x, y)\in \De \times \Omega} {\dist(x, y)}$.  
\end{proof}
\subsubsection*{Applications}
In \cite{Gigli2023On, mondino2022lipschitz}, 
the  authors adopt a perturbation argument of \cite{ZhangZhu18},  in the spirit of the classical Jensen's maximum principle and ideas from Petrunin \cite{Petrunin96} and Zhang--Zhu  \cite{ZhangZhu12},  to study harmonic maps from RCD metric measure spaces to ${\rm CAT}(0)$ spaces.
By Theorem \ref{th2},  we  get the following estimate which plays a key role in their perturbation argument.

% We remark that the strategy has strong analogies with the so-called two points maximum principle, see for instance \cite{Korevaar,Kruzkov} and \cite{Andrews} for a recent survey.

\begin{corollary}[cf. \cite{mondino2022lipschitz}, Theorem  4.3]\label{coro:mainperturbJensen}
Let $\ms$ be an   $\rcdkn$ metric measure space for some $K\in\R$ and $N\in (1, \infty)$. Let $\Omega\subset X$ be a bounded open domain with $\mm(\partial \Omega)=0$.   Assume that 
 $u\in \D({ \Delta}, \Omega) \cap C(\Omega)$ with
$
\Delta u\le L$
for some positive constant $L$.
Then, for any compact set $\De\subset X$ and for any $t>0$  so that 
\begin{equation*}
\Rem_2(\De,\Omega,u, t)\subset \Omega,
\end{equation*} 
or 
\[
  \Rem_1(\De, \Omega,  u, t) \subset \Omega,
\]
\[
\forall y\in \De,  ~\exists x\in  \Rem_1(\De, \Omega,  u, t), ~\dist(x, y)=t, ~\inf_{\overline \Omega}\big(u+\dist_y\big)=u(x)+\dist(x, y).
\]
It holds the following estimate:
\begin{equation}\label{eq:quantest}
\mm(\De)\le C(K,N, \De, \Omega,t,L)\, \mm\big( \Rem_i(\De, \Omega, u, t)\big)~~~i=1,2
\end{equation}
for some explicit constant
\[
C(K,N, \De, \Omega,t,L):=
\begin{cases}
\displaystyle \left(c_{K/N}(\Theta)+\frac{ts_{K/N}(\Theta)}{N\Theta} L\right)^N, & \textrm{if}\ K<0, \crcr
\displaystyle \left(1+\frac tN L\right)^N & \textrm{if}\ K =0,  \crcr
\displaystyle   \left(c_{K/N}(\Phi)+\frac{ts_{K/N}(\Phi)}{N\Phi}  L\right)^N & \textrm{if}\ K>0,
\end{cases}
\]
where $\Theta:=\sup_{(x, y)\in \De \times \Omega} {\dist(x, y)}$ and $\Phi:=\inf_{(x, y)\in \De \times \Omega} {\dist(x, y)}$.
In particular, if $\mm(\De)>0$, then $\mm\big( \Rem_i(\De, \Omega, u, t)\big)>0$.

\end{corollary}

\bigskip

On the basis of  our ABP estimate and the ideas of Cabr\'e \cite{Cabre-ABP} and Wang--Zhang \cite{WZ-ABP}, one can  also prove the Harnack inequality and certain geometric inequalities on metric measure spaces with suitable assumptions.  In this paper, we will  only  study an isoperimetric type inequality.

Let $\Omega \subset X$  be an open set. Recall that the \emph{upper Minkowski content}  is defined as
\[
\mm^+(\Omega):=\mathop{\limsup}_{\epsilon \downarrow 0} \frac{\mm(\Omega^\epsilon)-\mm(\Omega)}{\epsilon}
\]
where $\Omega^\epsilon \subset X$ is the $\epsilon$-neighbourhood of $\Omega$ defined as $\Omega^\epsilon:=\{x: \dist(x, \Omega)<\epsilon\}$.   In metric-measure setting, this notion plays the role of `boundary area' (cf. \cite{ADMG-P} for more discussions).

\begin{definition}[Uniform exterior sphere condition]
 Let $\Omega\subset  X$ be an open set. We say that $\Omega$ satisfies the \emph{uniform exterior
sphere condition}  if  there exists $r>0$ such that for all $x \in \partial\Omega$ there is $p_x \in \Omega^c$ such that $\dist(x, p_x)=r$ and $B_r(p_x)\subset \Omega^c$.
\end{definition}

\begin{definition}[$H$-mean convex]
 Let $\Omega$ be an open set in a  ${\rm CD}(0, N)$ space. Let $u$ be the signed distance function defined by
\begin{equation} 
u(x):=\left\{\begin{array}{ll}
\dist(x, \Omega), &\text{if}~~ x\in \Omega^c\\
-\dist(x, \Omega^c), &\text{if}~~x\in \Omega.
\end{array}
\right.
\end{equation} We say that $\partial \Omega$  is  $H$-mean convex if there is $\sigma >0$ so that  $u\in {\rm D}(\Delta,  \Omega_{\sigma}\setminus \overline{\Omega})$ and 
 ${\Delta} u\llcorner_{\Omega_{\sigma}\setminus \overline{\Omega}} \leq -H \mm\llcorner_{\Omega_{\sigma}\setminus \overline{\Omega}}$.

\end{definition}

 \brk
 By the representation formula for the Laplacian of the signed distance function,  proved by Cavalletti--Mondino \cite[Corollary 4.16]{CM-Laplacian} based on needle decomposition,  one can prove: if $\Omega$ satisfies uniform exterior sphere condition, and the outer mean curvature of $\partial \Omega$  is bounded from below by $H$ in the sense of Ketterer \cite{KettererHK},  then it is  $H$-mean convex.
 
 In particular, in the smooth setting, $\partial \Omega$ is $H$-mean convex  if and only if its (outer) mean curvature is bounded from below by $H$.
 \erk

 \medskip

Then we can prove a generalized Steiner-type formula.  We  refer the readers to \cite{KettererHK} for a   {H}eintze--{K}archer type inequality.

\begin{theorem}[Generalized Steiner-type formula]
Let $\ms$ be an ${\rm RCD}(0, N)$ metric measure space with $N\geq 2$.  Let $\Omega$ be a bounded open set satisfying uniform exterior sphere condition with $\mm(\partial \Omega)=0$.  Assume that $\partial \Omega$ is $H$-mean convex for some $H\geq 0$.  Then, as $\epsilon \to 0$,
\begin{equation*}
\mm(\Omega_{\epsilon}) \leq \mm(\Omega)+
\left(\epsilon-\frac 12 H\epsilon^2\right)\mm^+(\Omega)+o(\epsilon^2).
\end{equation*}

\end{theorem}
\begin{proof}
For $\epsilon \in (0, r/2)$ and $\delta\in  (0, \epsilon/2)$,  consider the contact set $\Rem_1(\overline{\Omega_{2\epsilon-\delta}}\setminus {\Omega_{\epsilon+\delta}},   \Omega_\epsilon\setminus \overline{\Omega}, u,\epsilon)$.
By uniform exterior sphere condition, we can see that $$\Rem_1(\overline{\Omega_{2\epsilon-\delta}}\setminus {\Omega_{\epsilon+\delta}},   \Omega_\epsilon\setminus \overline{\Omega}, u,\epsilon) \subset \Omega_\epsilon\setminus \overline{\Omega}$$
and
\[
\forall y\in \overline{\Omega_{2\epsilon-\delta}}\setminus {\Omega_{\epsilon+\delta}},  ~\exists x\in  \Rem_1, ~\dist(x, y)=\epsilon, ~\inf_{ \Omega_\epsilon\setminus \overline{\Omega}}\big(u+\dist_y\big)=u(x)+\dist(x, y).
\]

Applying Corollary \ref{coro:mainperturbJensen} to $u$,  we obtain
\begin{equation*}
\mm(\overline{\Omega_{2\epsilon-\delta}}\setminus {\Omega_{\epsilon+\delta}}) \leq \mm( \Omega_\epsilon \setminus \overline\Omega)\Big(1-\frac{\epsilon H}{N-1} \Big)^{N-1}.
\end{equation*}
Letting $\delta\to 0$ we get
\begin{equation*}
\mm({\Omega_{2\epsilon}}\setminus {\overline{\Omega_{\epsilon}}}) \leq \mm( \Omega_\epsilon \setminus \overline\Omega)\Big(1-\frac{\epsilon H}{N-1} \Big)^{N-1}
\end{equation*}
By Taylor's expansion  with respect to $\epsilon$,  we know
\[
\Big(1-\frac{\epsilon H}{N-1} \Big)^{N-1} \leq 1- H \epsilon+\frac {1}{2} H^2\epsilon^2.
\]
So
\begin{eqnarray*}
\mm(\Omega_{2\epsilon}\setminus \overline\Omega) &\leq& \left(2- H \epsilon+\frac {1}{2} H^2\epsilon^2\right)\mm(\Omega_{\epsilon}\setminus \overline\Omega)\\
&\leq& \left(2- H \epsilon+\frac {1}{2} H^2\epsilon^2\right)\left(2- \frac 12 H \epsilon+\frac {1}{2^3} H^2\epsilon^2\right)\mm(\Omega_{\epsilon/2}\setminus \overline\Omega)\\
&...& \\ 
&\leq& 2^n\left(1-H\epsilon+\frac 1{2^n}H\epsilon+O(\epsilon^2)\right)\mm(\Omega_{\epsilon/2^{n-1}}\setminus \overline\Omega).
\end{eqnarray*}
Letting $n\to \infty$ we get
\[
\mm(\Omega_{2\epsilon}\setminus \overline\Omega) \leq 
(2\epsilon-2H\epsilon^2)\mm^+(\Omega)+o(\epsilon^2)
\]
Replacing $\epsilon$ by $\epsilon/2$ we prove the theorem.
\end{proof}

\def\cprime{$'$}
\providecommand{\bysame}{\leavevmode\hbox to3em{\hrulefill}\thinspace}
\providecommand{\MR}{\relax\ifhmode\unskip\space\fi MR }
% \MRhref is called by the amsart/book/proc definition of \MR.
\providecommand{\MRhref}[2]{%
  \href{http://www.ams.org/mathscinet-getitem?mr=#1}{#2}
}
\providecommand{\href}[2]{#2}

\end{document}